\documentclass{amsart}

\usepackage{graphicx}
\usepackage{amsmath}
\usepackage{amsthm}
\usepackage{amsfonts}
\usepackage{amssymb}
\usepackage{url}
\usepackage{MnSymbol}
\usepackage{enumerate}

\newlength{\figheight}
\newtheorem{thm}{Theorem}[section]
\newtheorem{defn}[thm]{Definition}

\newtheorem{cor}[thm]{Corollary}
\newtheorem{lem}[thm]{Lemma}
\newtheorem{remark}[thm]{Remark}

\newtheorem{fact}[thm]{Fact}
\newtheorem{claim}[thm]{Claim}
\newtheorem{cons}[thm]{Construction}

\newcommand{\sat}[2]{{\rm sat}(#1,#2)}
\newcommand{\isat}[2]{{\rm indsat}(#1,#2)}
\newcommand{\indsat}[1]{#1\text{-induced-saturated}}

\newcommand{\comp}[1]{\overline{#1}}

\newcommand{\isatno}[1]{\left\lceil\frac{#1+1}{3}\right\rceil}
\newcommand{\flip}[2] {#1 \fivedots #2} 
\newcommand{\ex}{{\rm ex}}
\newenvironment{proofcite}[1]{\noindent{\bf Proof of #1.\,}}{\hfill$\Box$}

\newcommand{\noi}{\noindent}

\title{Induced Saturation Number}
\author{Ryan R. Martin}
\address{Department of Mathematics, Iowa State University, Ames, Iowa 50011}
\email{rymartin@iastate.edu}
\thanks{The first author's research partially supported by NSF grant DMS-0901008 and by an Iowa State University Faculty Professional Development grant.}
\author{Jason J. Smith}
\address{Department of Mathematics, Iowa State University, Ames, Iowa 50011}
\email{smithj@iastate.edu}
\thanks{The second author's research partially supported from NSF grant DMS-0901008.}

\begin{document}
\maketitle

\begin{abstract}
In this paper, we discuss a generalization of the notion of saturation in graphs in order to deal with induced structures.  In particular, we define ${\rm indsat}(n,H)$, which is the fewest number of gray edges in a trigraph so that no realization of that trigraph has an induced copy of $H$, but changing any white or black edge to gray results in some realization that does have an induced copy of $H$.

We give some general and basic results and then prove that ${\rm indsat}(n,P_4)=\lceil (n+1)/3\rceil$ for $n\geq 4$ where $P_4$ is the path on $4$ vertices.  We also show how induced saturation in this setting extends to a natural notion of saturation in the context of general Boolean formulas.
\end{abstract}


\section{Introduction}

\subsection{Saturation in graphs}
A graph $G$ is \textit{$H$-saturated} if it does not contain $H$ as a subgraph, but $H$ occurs whenever any new edge is added to $G$. The Tur\'an type problems deal with $H$-saturated graphs, in particular $\ex(n;H)=\max\{|E(G)| : |V(G)|=n, G \text{ is } H\text{-saturated}\}$. Although previous work has been done in the field, $\sat{n}{H}=\min\{|E(G)| : |V(G)|=n, G \text{ is } H\text{-saturated}\}$ was first formally defined by K\'aszonyi and Tuza in \cite{KaszonyiTuza86}. They also find the saturation number for paths, stars, and matchings.  In particular they find
$$ \sat{n}{P_3}=\left\lfloor\frac{n}{2}\right\rfloor, \qquad \sat{n}{P_4}=
 \left\{\begin{array}{ll}                                                                                                k, & \hbox{if $n=2k$;} \\                                                                                                k+1, & \hbox{if $n=2k-1$,}                                                                                              \end{array}\right. \qquad \sat{n}{P_5}=n-\left(\left\lfloor\frac{n-2}{6}\right\rfloor+1\right) , $$
$$ \sat{n}{P_h}=n-\left\lfloor\frac{n}{3\cdot2^{k-1}-1}\right\rfloor, \mbox{when $h=2k\ge 6$,}\quad\mbox{and}\quad \sat{n}{P_h}=n-\left\lfloor\frac{n}{2^{k+1}-2}\right\rfloor, \mbox{when $h=2k+1\ge 7$} . $$

They also find a general upper bound for saturation number. That is for any graph $H$ there exists a constant $c=c(H)$, such that $\sat{n}{H}<cn$.

Pre-dating~\cite{KaszonyiTuza86}, Erd\H{o}s, Hajnal, and Moon in 1964 found the saturation number of complete graphs in \cite{ErdosHajnalMoon64}; that is, $\sat{n}{K_h}=(h-2)n-\binom{h-1}{2}$. Eight years later in~\cite{Ollmann72}, Ollmann found the saturation number of the four cycle to be $\sat{n}{C_4}=\lceil 3(n-5)/2\rceil$.  More recently, R. Faudree, Ferrara, Gould and Jacobson~\cite{Gouldetal09} found the saturation number of $tK_p$ and Chen~\cite{Chen09} found the saturation number of $C_5$.  A more complete background of known saturation results is provided in the dynamic survey by J. Faudree, R. Faudree and Schmitt~\cite{FaudreeSchmitt11}.

\subsection{Notation and definitions}
We would like to generalize the notion of saturation to include induced subgraphs.  To do this we use definitions given by Chudnovsky~\cite{Chudnovsky06}. Let a \textit{trigraph}\footnote{Doug West has suggested the term ``fuzzy graphs'' to comport with the notion of fuzzy sets. However, Rosenfeld~\cite{Rosenfeld} has used this term for a different object and so we choose to use Chudnovsky's notation.} $T$ be a quadruple $(V(T),EB(T),EW(T),EG(T))$, with $V(T)$ being the vertices of the trigraph, $EB(T)$ being the black edges, $EW(T)$ being the white edges, and $EG(T)$ being the gray edges.  We will think of these as edges, non-edges, and `free' edges, where `free' means those edges can either be black or white.  We note here that if $EG(T)=\emptyset$, then our trigraph is just a graph.

\begin{defn}
A \textbf{realization} of a trigraph $T$ is a graph $G$ with $V(G)=V(T)$ and $E(G)=EB(T)\cup S$ for some subset $S$ of $EG(T)$.  That is, we set some gray edges to black edges and the remaining gray edges to white.

For any graph $H$, we say that a trigraph $T$ \textbf{has a realization of $H$} if there is a realization of $T$ which has $H$ as an induced subgraph.

A trigraph $T$ is $\indsat{H}$ if no realization of $T$ contains $H$ as an induced subgraph, but $H$ occurs as an induced subgraph of some realization whenever any black or white edge of $G$ is changed to gray. The \textbf{induced saturation number} of $H$ with respect to $n$ is defined to be $\isat{n}{H}=\min\{|EG(T)| : |V(T)|=n, T \text{ is } \indsat{H}\}$.
\end{defn}

\begin{remark}By definition, the only trigraphs on fewer than $|V(H)|$ vertices that are $\indsat{H}$ are those in which all edges are gray.\end{remark}

In colloquial terms, to find $\isat{n}{H}$ is to find a trigraph $T$ on $n$ vertices with the fewest number of gray edges having the property that there is no induced subgraph $H$ but if any black or white edge is changed to gray, then we find $H$ as an induced subgraph.

If an edge has endvertices $v$ and $w$, we denote it $vw$ or $wv$.  For a vertex $v$ in a trigraph $T$, the \textit{white neighborhood} of $v$ is the set $\{w : wv\in EW(T)\}$. The \textit{black neighborhood} and \textit{gray neighborhood} are defined similarly. In addition, we write $v\sim w$ to mean that $vw\in EB(T)$ and $v\not\sim w$ to mean that $vw\in EW(T)$. In a graph, we say that vertices $v_1v_2\cdots v_k$ \textit{form a path} if $v_iv_{i+1}$ is an edge for $i=1,\ldots,k-1$ but all other pairs of these vertices are nonedges.

Since induced saturation involves taking a trigraph $T$ and changing an otherwise nongray edge $e$ to gray, we denote $\flip{T}{e}$ to be the trigraph that results from making $e$ a gray edge.  This corresponds with the ``$+$'' notation in ordinary saturation.

The \textit{complement}, $\comp{T}$ of a trigraph $T$ is a trigraph with $V(\comp{T})=V(T)$, $EB(\comp{T})=EW(T)$, $EG(\comp{T})=EG(T)$, and $EW(\comp{T})=EB(T)$, which extends the definition of graph complement.  We define a \textit{gray component} to be a set of vertices that are connected in the graph $(V(T),EG(T))$.  We do not require that every edge within the component is gray, just that there is a gray path connecting any pair of vertices.  Analogously, we define \textit{white}, \textit{black}, \textit{black/gray}, and \textit{white/gray components}.

If $V_1$ and $V_2$ are disjoint sets in $V(T)$, then $T[V_1,V_2]$ denotes the set of edges with one endpoint in $V_1$ and the other in $V_2$ and $T[V_1]$ denotes the subtrigraph induced by $V_1$.

\subsection{Main Result}
In Section~\ref{sec:basic}, we establish a few very basic results regarding induced saturation, such as the fact that $\isat{n}{H}$ is bounded by $\sat{n}{H}$, but the main result is the induced saturation number for paths of length 4.

\begin{thm}\label{TheoremISatNoForP4} $\isat{n}{P_4}=\isatno{n}$, for all $n\ge 4$. \end{thm}

The proof of this result comprises the rest of the paper. Section~\ref{sec:upperbd} gives the upper bound $\isat{n}{P_4}\leq\isatno{n}$.  Section~\ref{sec:lowerbd} gives the lower bound $\isat{n}{P_4}\geq\isatno{n}$.  Section~\ref{sec:facts} gives a number of general facts regarding $\indsat{P_4}$ trigraphs which are used repeatedly in the subsequent sections. Section~\ref{sec:mainproof} is the main proof of the lower bound. Section~\ref{sec:proofgrayexists} contains the proof that a $\indsat{P_4}$ trigraph contains at least one gray edge and Section~\ref{sec:prooftechlemma} contains the proof of an important technical lemma. Section~\ref{sec:factproofs} contains the proofs of the facts enumerated in Section~\ref{sec:facts}. Section~\ref{sec:conc} contains some concluding remarks.

\subsection{Generalizations and Applications}

Although we are concerned with induced graph saturation, there is a more general view, in terms of satisfiability.  Given a disjunctive normal form (DNF), we want to find a partial assignment of variables such that (1) there is no way to complete the assignment to a true one but (2) if any one of the assigned variables were unassigned, then this new partial assignment can be completed to a true one.

In the specific case of graphs, the DNF constructed from the set of pairs of $n$ vertices is one comprised of clauses, each of which represents an instance of a potential induced copy of $H$. For instance, if $e_1,\ldots,e_6$ represent unordered pairs such that $e_1,e_2,e_3$ being edges and $e_4,e_5,e_6$ being nonedges induces $P_4$, then the corresponding clause is $x_1\wedge x_2\wedge x_3\wedge\overline{x_4}\wedge\overline{x_5}\wedge\overline{x_6}$, where the variable $x_i$ corresponds to the pair $e_i$, for $i=1,\ldots,6$.

The trigraph has a number of applications related to Szemer\'edi's Regularity Lemma~\cite{RegLem} (see also~\cite{RegLemSurvey1,RegLemSurvey2}). A trigraph can also be thought of as a reduced graph in which the a black edge represents a pair with density close to $1$, a white edge represents a pair with density close to $0$ and a gray edge represents a pair with density neither near $1$ nor near $0$. Such a configuration is used in a number of applications of the regularity lemma related to induced subgraphs.  See, for instance,~\cite{BaloghMartin}.


\section{Basic Results}
\label{sec:basic}
In this section we establish some basic results on the induced saturation number.  We find a relationship between induced saturation and classical saturation as well as establish the induced saturation number for a few families of graphs.

The first result presented relates the saturation number to the induced saturation number.

\begin{thm}
$\isat{n}{H}\le\sat{n}{H}$ for all graphs $H$ and all positive integers $n\geq |V(H)|$.
\end{thm}
\begin{proof} We can take the graph that gives the upper bound for saturation number and change all the edges to gray edges and leave all non-edges white.  Then clearly changing any white edge to gray forms an induced graph $H$ in the same way as adding any black edge created an $H$ in the case of saturation of graphs.
\end{proof}

For the complete graph, the numbers are identical:
\begin{thm}
If $K_h$ is an $h$-vertex complete graph, then $\isat{n}{K_h}=\sat{n}{K_h}=(h-2)n-\binom{h-1}{2}$ for all $n\geq h\geq 3$.
\end{thm}

\begin{proof}
This is a direct result of the fact that $K_h$ being an induced subgraph is identical to $K_h$ being a subgraph. All $\indsat{K_h}$ trigraphs have no black edges.
\end{proof}

One may wonder if it is ever the case that $\isat{n}{H}<\sat{n}{H}$.  That is, can we improve on the bound provided by the saturation number? Theorem~\ref{CompleteWithoutEdge} shows that the answer is yes.  We know from K\'aszonyi and Tuza in \cite{KaszonyiTuza86} that $\sat{n}{P_3}=\left\lfloor\frac{n}{2}\right\rfloor$ and, trivially, $\sat{n}{K_h^{-}}\ge\left\lfloor\frac{n}{2}\right\rfloor$, for $h\ge4$.

\begin{thm}\label{CompleteWithoutEdge}
Let $K_h^{-}$ denote a graph on $h$ vertices with exactly one nonedge. Then, $\isat{n}{K_h^{-}}=0$ for all $n\geq h$.
\end{thm}

\begin{proof}
For an upper bound, we choose our trigraph $T$ to be the complete graph of black edges on $n$ vertices.  Clearly, $T$ does not contain an induced $K_h^{-}$.  Now, choose an arbitrary edge, $e$, in $T$, and consider $\flip{T}{e}$. Since all edges were black, the new gray edge can be used as the nonedge in a copy of $K_h^{-}$.
\end{proof}

Since $P_3$, the path on $3$ vertices, is the same as $K_3^{-}$, $\isat{n}{P_3}=0$ as well.
\begin{cor}
$\isat{n}{P_3}=0$, for all $n\geq 3$.
\end{cor}

So, not only can the induced saturation number differ drastically from $K_h$ to $K_h^{-}$, but it can be much less than the saturation number for the same graph. Theorem~\ref{TheoremISatNoForP4} establishes that the induced saturation number $\isat{n}{P_4}$ is both nonzero for $n\geq 2$ and is strictly less than $\sat{n}{P_4}$ for $n=5$ and $n\geq 7$.



\section{Proof of upper bound in the main result}
\label{sec:upperbd}

In this section, we prove the following via constructions.
\begin{lem} $\isat{n}{P_4}\le\lceil\frac{n+1}{3}\rceil$ for $n\geq 4$.
\label{lem:upperbd}
\end{lem}

\begin{proof}
We have three constructions depending on the remainder of $n$ upon division by $3$. We will denote the trigraph construction on $n$ vertices by $T_n'$.

\begin{cons} The trigraph $T_n'$ for $n\geq 4$ has the following vertex set:
$$ V(T_n')=\left\{\begin{array}{ll}
                  \bigcup_{i=1}^k\{a_i,b_i,c_i\}\cup\{a_{k+1},b_{k+1}\}, & \mbox{if $n=3k+2$;} \\
                  \bigcup_{i=1}^k\{a_i,b_i,c_i\}\cup\{a_{k+1},b_{k+1}\}\cup\{c_0\}, & \mbox{if $n=3k+3$;} \\
                  \bigcup_{i=1}^k\{a_i,b_i,c_i\}\cup\{a_{k+1},b_{k+1}\}\cup\{a_0,b_0\}, & \mbox{if $n=3k+4$;} \end{array}\right. $$
The edges are colored as follows:
$$ \begin{array}{rl}
      a_ib_i\in EG(T_n'), & \forall i; \\
      c_0c_1\in EG(T_n'), & \mbox{if $n\equiv 0\bmod{3}$;} \\
      a_ia_j,a_ib_j,b_ib_j\in EW(T_n'), & \forall i\neq j; \\
      c_ic_j\in EB(T_n'), & \forall i\neq j, \{i,j\}\neq\{0,1\}; \\
      c_ia_j,c_ib_j\in EB(T_n'), & \forall i\leq j; \\
      c_ia_j,c_ib_j\in EW(T_n'), & \forall i>j. \end{array} $$
\end{cons}

\begin{remark} The subtrigraph of $T_n'$ induced by $\bigcup_{i=1}^k\{a_i,b_i,c_i\}$ resembles the so-called half-graph. In fact $\bigcup_{i=1}^k\{a_i,c_i\}$ is the half-graph on $2k$ vertices if black edges are interpreted as the edges of the graph and white edges are interpreted as the nonedges.
\end{remark}

The various constructions are shown in Figures~\ref{P4ConstructionTwo},~\ref{P4ConstructionZero} and~\ref{P4ConstructionOne}.  It is easy to compute that, in each case, there are exactly $\left\lceil\frac{n+1}{3}\right\rceil$ gray edges in each $T_n'$ for $n\geq 4$.

\begin{figure}[ht]
   \centering
   \includegraphics[height=\figheight]{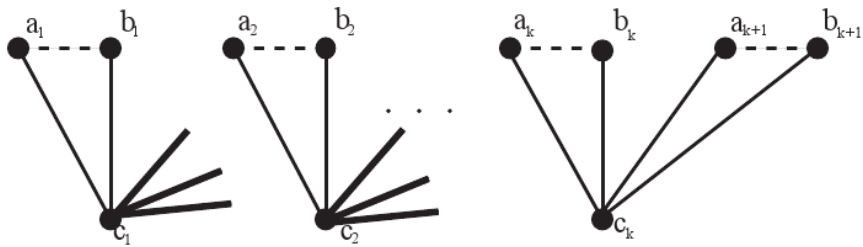}
   \caption{$T_n'$ when $n=3k+2$}
   \label{P4ConstructionTwo}
\end{figure}
\begin{figure}[ht]
   \centering
   \includegraphics[height=\figheight]{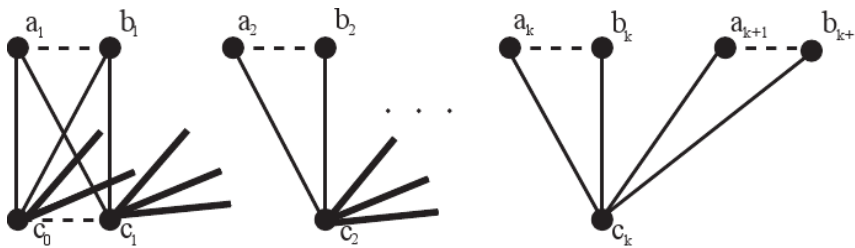}
   \caption{$T_n'$ when $n=3k+3$}
   \label{P4ConstructionZero}
\end{figure}
\begin{figure}[ht]
   \centering
   \includegraphics[height=\figheight]{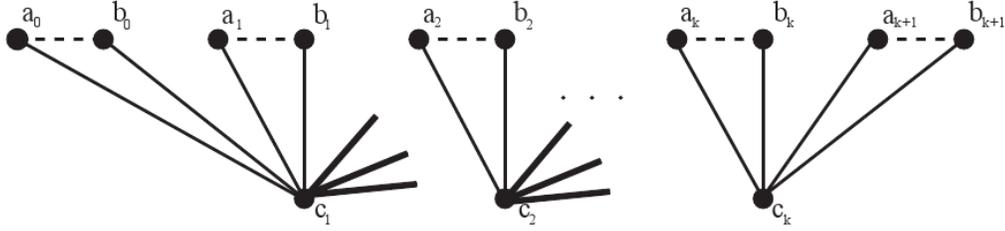}
   \caption{$T_n'$ when $n=3k+4$}
   \label{P4ConstructionOne}
\end{figure}

It remains to show that, for $n\geq 4$, the graph $T_n'$ is $\indsat{P_4}$.  Recall to do this, first we must show that $T$ does not contain a realization which has a $P_4$ as an induced subgraph.  Second, we must show that if we flip any white or black edge to gray, then $T$ does contain a realization which has a $P_4$ as an induced subgraph.  We leave it to the reader to verify this in the case of $n=4$, in which $T_4'$ is two disjoint gray edges and a white 4-cycle. So, we may assume $n\geq 5$, i.e., $k\geq 1$.



First, we show that $T_n'$ does not contain an induced $P_4$.  If $n=3k+3$, then if $c_0$ is in such a realization, then so must $c_1$ because $c_0x\in EB(T_n')$ if $x\neq c_1$.  But $\{c_0,c_1,x,y\}$ must contain a black $C_4$ for all distinct $x,y\not\in\{c_0,c_1\}$, hence no $P_4$ realization is possible.  So, $c_0$ is in no realization of a $P_4$.  In addition, no realization of $P_4$ can have $3$ or $4$ vertices from $\{c_1,\ldots,c_k\}$ because they induce a black $C_3$.  If we have four distinct vertices $\{c_i,c_j,x,y\}$, where $1\leq i<j\leq k$, then there is either a black $C_3$ or a white $C_4$, hence no realization of $P_4$.  If we have no vertices from $\{c_1,\ldots,c_k\}$, then we have a white $C_4$ and so no realization of $P_4$. Finally, consider four distinct vertices $\{c_i,x,y,z\}$ where $x,y,z\not\in\{c_1,\ldots,c_k\}$.  We may assume that two of $\{x,y,z\}$ are adjacent via a gray edge, otherwise there is a white $C_3$. Without loss of generality, the vertices are of the form $\{c_i,a_j,b_j,a_{\ell}\}$. Regardless of the values of $j$ and $\ell$ relative to $i$, there is either a black star with $3$ leaves, a white star with $3$ leaves or a white $C_4$, all of which preclude a realization of $P_4$.

Second, we show that if any black or white edge in $T_n'$ is flipped to gray then the new trigraph does contain an induced $P_4$.  We will take care of this with case analysis, focusing primarily on Figure~\ref{P4ConstructionTwo} first and then filling in the missing edges for the other two constructions.  One observation we will make is that if trigraph $T$ contains no realization of $P_4$ but $\flip{T}{e}$ does, for some white edge $e$, then the realization must use $e$ as a black edge.  The complementary observation (flipping a black edge requires it be used as a white edge) holds as well.

\begin{enumerate}
   \item \textbf{Flip $a_ia_j$:} We consider $\flip{T_n'}{a_ia_j}$ for $i<j$, which flips $a_ia_j$ from white to gray. In this case we have an induced $P_4$ with vertices $b_ia_ia_jb_j$. By symmetry, we conclude that $\flip{T_n'}{a_ib_j}$ and $\flip{T_n'}{b_ib_j}$ have a realization of $P_4$, for all $i\neq j$.

   \item \textbf{Flip $c_ic_j$:} We consider $\flip{T_n'}{c_ic_j}$ for $i<j$ and $\{i,j\}\neq\{0,1\}$, which flips $c_ic_j$ from black to gray. In this case we have an induced $P_4$ with vertices $a_ic_ia_jc_j$.

   \item \textbf{Flip $c_ia_i$:} We consider $\flip{T_n'}{c_ia_i}$, which flips $c_ia_i$ from black to gray. In this case we have an induced $P_4$ with vertices $a_ib_ic_ia_{i+1}$. By symmetry, we conclude that $\flip{T_n'}{c_ib_i}$ has a realization of $P_4$, for all $i$.

   \item \textbf{Flip $c_ia_j$, $i<j$:} We consider $\flip{T_n'}{c_ia_j}$ for $i<j$, which flips $c_ia_j$ from black to gray. In this case we have an induced $P_4$ with vertices $a_ic_ib_ja_j$. By symmetry, we conclude that $\flip{T_n'}{c_ib_j}$ has a realization of $P_4$, for all $i<j$.

   \item \textbf{Flip $c_ia_j$, $i>j$:} We consider $\flip{T_n'}{c_ia_j}$ for $i>j$, which flips $c_ia_j$ from white to gray. In this case we have an induced $P_4$ with vertices $b_ja_jc_ia_i$. By symmetry, we conclude that $\flip{T_n'}{c_ib_j}$ has a realization of $P_4$, for all $i>j$.
\end{enumerate}

Therefore, $T_n'$ is $\indsat{P_4}$ for all $n\geq 5$ and, together with the case $n=4$, we are done.
\end{proof}


\section{Proof of lower bound in the main result}
\label{sec:lowerbd}

In this section, we prove the lower bound found in Theorem~\ref{TheoremISatNoForP4}:
\begin{thm}\label{TheoremP4Smallest} $\isat{n}{P_4}\ge\lceil\frac{n+1}{3}\rceil$. \end{thm}

The proof is done by strong induction and bulk of the work will be done by the technical lemma, Lemma~\ref{TechnicalLemma}. Before the main proof, some basic facts will be enumerated in Section~\ref{sec:facts}. Recall that a trigraph $T$ that is $\indsat{P_4}$ with $|V(T)|\leq 3$ is a complete gray graph.  To ease the statements of the proof, we state a few definitions:

Let $T$ be a trigraph. A star in some color class of $T$ is a \textit{trivial star} if it has two (or fewer) vertices.  If a star has three or more vertices, it is called a \textit{nontrivial star}. A connected component of a color class of $T$ is a \textit{trivial component} if it consists of a single vertex. If it consists of at least two vertices, it is called a \textit{nontrivial component}.


\subsection{General Facts}
\label{sec:facts}

The proofs of the following facts will appear in Section~\ref{sec:factproofs}.  We state them here as they provide a feel for how the proof of the main result will be handled.

\begin{fact}\label{SelfComplementary} If a trigraph $T$ is $\indsat{P_4}$, then the complement of $T$, $\comp{T}$, is also $\indsat{P_4}$.
\end{fact}

\begin{fact} \label{CompleteDisjoint} Let $T$ be a trigraph which is $\indsat{P_4}$. If $V(T)$ is partitioned into $V_1$ and $V_2$ such that either each edge in $T[V_1,V_2]$ is white or each edge in $T[V_1,V_2]$ is black, then both $T[V_1]$ and $T[V_2]$ are $\indsat{P_4}$.
\end{fact}

Let $T$ be a trigraph with distinct vertices $v_1$ and $v_2$ and $S\subseteq V(T)-\{v_1,v_2\}$. We say \textit{$v_1$ and $v_2$ have the same neighborhood in $S$} if $v_1s$ and $v_2s$ have the same color for every $s\in S$.

\begin{fact} \label{FactSameBehavior} Let $T$ be a trigraph that is $\indsat{P_4}$. Let $S$ be a subset of $V=V(T)$ such that $T[S,V-S]$ has no gray edges. If every pair of vertices in $S$ has the same neighborhood in $V-S$, then $S$ is $\indsat{P_4}$.
\end{fact}

By Fact~\ref{FactNoGrayP4}, the gray components of size at least 2 in $\indsat{P_4}$ trigraphs are either $K_3$'s or $S_k$'s for $k\geq 2$.
\begin{fact} \label{FactNoGrayP4}
If $T$ is a $\indsat{P_4}$ trigraph on $n\geq 2$ vertices, then each nontrivial gray component in $T$ is either a complete gray $K_3$ or a gray star, $S_k$, $k\ge2$, where $S_k$ denotes the star on $k$ vertices.
\end{fact}

Facts~\ref{StarBehavior} and~\ref{TriangleBehavior} establish a partition of $V(T)$ according to how those vertices relate to a particular gray component. Note that Fact~\ref{StarBehavior} also applies if the gray component is an edge on 2 vertices.
\begin{fact} \label{StarBehavior} Let $T$ be a trigraph that is $\indsat{P_4}$. If $k\geq 2$ and $\{u,v_1,v_2,\dots,v_{k-1}\}$ is a gray star in $T$ with center $u$, then $V(T)-\{u,v_1,v_2,\dots,v_{k-1}\}$ can be partitioned into $X$, $Y$ and $Z$ such that the following occur:
\begin{itemize}
   \item The edges $xu$ and $xv_i$ are black for all $i\in\{1,\ldots,k-1\}$ and for all $x\in X$.
   \item The edges $yu$ and $yv_i$ are white for all $i\in\{1,\ldots,k-1\}$ and for all $y\in Y$.
   \item One of the following occurs:
   \begin{itemize}
      \item The edges $v_iv_j$ and $v_iz$ are white and $uz$ are black for all distinct $i,j\in\{1,\ldots,k-1\}$ and for all $z\in Z$.
      \item The edges $v_iv_j$ and $v_iz$ are black and $uz$ are white for all distinct $i,j\in\{1,\ldots,k-1\}$ and for all $z\in Z$.
   \end{itemize}
\end{itemize}
\end{fact}
Note that in the case where $k=2$, Fact~\ref{StarBehavior} still holds, but the subcases for the behavior of vertices in $Z$ are equivalent, depending on the labeling of the vertices. Regardless, if $uv$ is a gray component, we may assume that the edges in $T[Z,\{u\}]$ are black and the edges in $T[Z,\{v\}]$ are white.

\begin{fact} \label{TriangleBehavior} Let $T$ be a trigraph that is $\indsat{P_4}$. If $\{v_1,v_2,v_3\}$ is a gray triangle in $T$, then $V(T)-\{v_1,v_2,v_3\}$ can be partitioned into $X$ and $Y$ such that the following occur:
\begin{itemize}
   \item The edges $xv_i$ are black for all $i\in\{1,2,3\}$ and for all $x\in X$.
   \item The edges $yv_i$ are white for all $i\in\{1,2,3\}$ and for all $y\in Y$.
\end{itemize}
\end{fact}~\\

\subsection{Proof of Theorem~\ref{TheoremP4Smallest}}
\label{sec:mainproof}

We will prove Theorem \ref{TheoremP4Smallest} by strong induction on $n$.  The inductive hypothesis is: If $T$ is a $\indsat{P_4}$ trigraph on $n\geq 2$ vertices, then $|EG(T)|\geq\isatno{n}$.



Note that the statement is true for the trivial base cases of $n=2,3$. So, let $T$ be a trigraph on $n\geq 4$ vertices that is $\indsat{P_4}$.  First, we use Lemma~\ref{LemGrayExists}, proven in Section~\ref{sec:factproofs} to establish that there must be at least one gray edge.
\begin{lem} \label{LemGrayExists}
If $T$ is a $\indsat{P_4}$ trigraph on $n\geq 2$ vertices, then $T$ has a gray edge.
\end{lem}

By Fact~\ref{FactNoGrayP4}, we know that gray edges must be in components that are either triangles or stars and by Facts~\ref{StarBehavior} and~\ref{TriangleBehavior}, each nontrivial gray component $C$ partitions $V(T)$ into sets $C$, $X$, $Y$ and $Z$, any of which could be empty.

\begin{claim}\label{cl:XYZ}
   With $X,Y,Z$ defined as above, the edges in $T[Z,X]$ are black and the edges in $T[Z,Y]$ are white.
\end{claim}
\begin{proofcite}{Claim~\ref{cl:XYZ}} If $C$ is a gray triangle or $Z$ is otherwise empty, the claim is vacuous, so we assume that $C$ is a star and $Z$ is not empty. Without loss of generality, let us assume $uv$ is a gray edge in the star such that the edges in $T[Z,\{u\}]$ are black and the edges in $T[Z,\{v\}]$ are white. Let $z\in Z$.  If $x\in X$ and $zx$ is white, then $zuxv$ has a realization of $P_4$. If $y\in Y$ and $zy$ is black, then $yzuv$ has a realization of $P_4$. This proves Claim~\ref{cl:XYZ}.
\end{proofcite}

Let $C_0$ be a gray component in which $Z$ is maximum-sized.
\begin{claim}\label{cl:nograyedge}
   With $X,Y,Z$ defined by $C_0$, there are no gray edges in $T[X,Y]$.
\end{claim}
\begin{proofcite}{Claim~\ref{cl:nograyedge}} Let $xy$ be a gray edge such that $x\in X$ and $y\in Y$.  By Claim~\ref{cl:XYZ}, the set $S=Z\cup\{u,v\}$ has the property that the edges in $T[\{x\},Z]$ are black and the edges in $T[\{y\},Z]$ are white.  Hence, the gray component that $xy$ is in has a larger $Z$-set than that formed by $C_0$.
\end{proofcite}

With $C_0$, $X$, $Y$ and $Z$ defined as above, we can turn to Lemma~\ref{TechnicalLemma}:
\begin{lem}\label{TechnicalLemma}
Let $T$ be a $\indsat{P_4}$ trigraph and let $X,Y$ be disjoint subsets of $V=V(T)$, $X\cup Y\neq\emptyset$ such that the following is true:
\begin{enumerate}
   \item $T[X,Y]$ has no gray edge.\label{tech:nograyedge}
   \item Each edge in $T[X,V-(X\cup Y)]$ is black and each edge in $T[Y,V-(X\cup Y)]$ is white.\label{tech:blackwhite}
   \item There exist vertices $u,v\in V(T)-(X\cup Y)$ such that $uv$ is a gray edge.\label{tech:grayedge}
\end{enumerate}

If $|X|+|Y|\geq 2$, then the number of gray edges in $T[X\cup Y]$ is at least $\left\lceil\frac{|X|+|Y|}{3}\right\rceil$.
\end{lem}

To see that we can apply Lemma~\ref{TechnicalLemma} to finish the proof, note the following: Condition (\ref{tech:nograyedge}) holds by Claim~\ref{cl:nograyedge}. Condition (\ref{tech:blackwhite}) holds because of the definition of $X$ and $Y$ from Facts~\ref{StarBehavior} and~\ref{TriangleBehavior} as well as because of Claim~\ref{cl:nograyedge}. Condition (\ref{tech:grayedge}) holds because of the existence of $C_0$.~\\

Let $uv$ be a gray edge in $C_0$ such that the edges in $T[Z,\{u\}]$ are black and the edges in $T[Z,\{v\}]$ are white.

\noi\textbf{Case (a.) $|Z|=1$}:\\
We will show that this case cannot occur.  Suppose it does and let $Z=\{z\}$. Consider a realization of $P_4$, call it $P$, in the trigraph $\flip{T}{zv}$. By the definition of $P$, it must use $zv$ as a black edge.  Then $P$ cannot contain an $x\in X$ because if it did, both $xz$ and $xv_2$ would be black, creating a black triangle and preventing $P$ from being a realization of $P_4$. But, $P$ cannot contain a member of $Y$ either, because there is no black/gray path from a vertex in $Y$ to either $z$ or $v$ that avoids $X$.

Since $P$ cannot contain a vertex in $X\cup Y$, it must be contained in the vertices $\{z,u,v\}$, but there are only three of those, a contradiction. So, Case (a.) cannot occur.\\

\noi\textbf{Case (b.) $Z=\emptyset$, $|X|+|Y|\leq 1$}:\\
In this case, $|V(T)|\leq|C_0|+1$.  If $T$ has at least $4$ vertices, then $C_0$ has at least $|C_0|-1=n-2$ gray edges, because $C_0$ is a gray component.  Since $n-2\geq\lceil (n+1)/3\rceil$ for all $n\geq 4$, Case (b.) satisfies the conditions of the theorem.\\

\noi\textbf{Case (c.) $Z=\emptyset$, $|X|+|Y|\geq 2$}:\\
We can apply Lemma~\ref{TechnicalLemma}. The number of gray edges in $T$ is at least
$$ \left\lceil\frac{|X|+|Y|}{3}\right\rceil+|C_0|-1 =\left\lceil\frac{n+2|C_0|-3}{3}\right\rceil \geq\left\lceil\frac{n+1}{3}\right\rceil . $$
Thus, Case (c.) satisfies the conditions of the theorem.~\\

\noi\textbf{Case (d.) $|Z|\geq 2$, $|X|+|Y|\leq 1$}:\\
By Fact~\ref{FactSameBehavior} (where $S=Z$) the subtrigraph $T[Z]$ is $\indsat{P_4}$. By the inductive hypothesis, the number of gray edges in $T[Z]$ is at least $\left\lceil\frac{|Z|+1}{3}\right\rceil$.
Hence, the number of gray edges in $T$ is at least
$$ \left\lceil\frac{|Z|+1}{3}\right\rceil+|C_0|-1 =\left\lceil\frac{|Z|+3|C_0|-2}{3}\right\rceil \geq\left\lceil\frac{(n-|C_0|-1)+3|C_0|-2}{3}\right\rceil =\left\lceil\frac{n+2|C_0|-3}{3}\right\rceil \geq\left\lceil\frac{n+1}{3}\right\rceil . $$
Thus, Case (d.) satisfies the conditions of the theorem.~\\

\noi\textbf{Case (e.) $|Z|\geq 2$, $|X|+|Y|\geq 2$}:\\
Again, by Fact~\ref{FactSameBehavior} the subtrigraph $T[Z]$ is $\indsat{P_4}$. We may apply both the inductive hypothesis and Lemma~\ref{TechnicalLemma}. The number of gray edges in $T$ is at least
$$ \left\lceil\frac{|Z|+1}{3}\right\rceil+\left\lceil\frac{|X|+|Y|}{3}\right\rceil+|C_0|-1 =\left\lceil\frac{|X|+|Y|+|Z|+1+3|C_0|-3}{3}\right\rceil =\left\lceil\frac{n+2|C_0|-2}{3}\right\rceil \geq\left\lceil\frac{n+2}{3}\right\rceil. $$
Thus, Case (e.) satisfies the conditions of the theorem. So, the proof of Theorem~\ref{TheoremP4Smallest} is complete.\hfill~$\Box$~\\


\subsubsection{Proof of Lemma~\ref{LemGrayExists}}
\label{sec:proofgrayexists}
We will prove that every $\indsat{P_4}$ trigraph on $n\geq 2$ vertices has at least one gray edge by strong induction on $n$. The base cases of $n=2,3$ are true by definition.

Let $n\geq 4$. For the sake of contradiction, assume $T$ is a $\indsat{P_4}$ trigraph on vertex set $V$, $|V|\geq 4$, which does not contain a gray edge. Since $T$ is assumed to have no gray edges, we may regard it as a graph and will sometimes use the language of edges and nonedges.

There must be an induced $P_3$ in $T$, otherwise $T$ is disjoint cliques.  If $T$ were disjoint cliques, then Fact~\ref{CompleteDisjoint} implies that each of the cliques is $\indsat{P_4}$. We may apply the inductive hypothesis unless each clique is of size 1 or there is one clique of size $n$. It is trivial to see that neither case is possible because a trigraph on at least 4 vertices with all white edges or with all black edges cannot be $\indsat{P_4}$.

Let the induced $P_3$ in $T$ be $uxv$. Let the set $X$ be the vertices that are adjacent to both $u$ and $v$. The set $Z_1$ will be the vertices that are adjacent to $u$ but not $v$. Likewise, the set $Z_2$ will be the vertices that are adjacent to $v$ but not $u$. Finally, let the set $Y$ be all vertices that are not adjacent to $u$ or $v$. See Figure~\ref{OneGraySetUp}.

\begin{figure}[ht]
     \centering
     \includegraphics[height=\figheight]{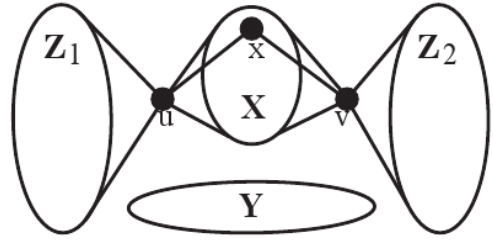}
     \caption{General set-up for Proof of Lemma \ref{LemGrayExists}. Trigraph $T$ containing a $P_3$ with vertices $uxv$.}\label{OneGraySetUp}
\end{figure}

First, note that $x\in X$, hence $X\not=\emptyset$.


Second, we show that all edges in $T[X,Z_1]$ and in $T[X,Z_2]$ must be black. To see this, let $x\in X, z_1\in Z_1,z_2 \in Z_2$. Without loss of generality, assume $xz_1\not\in EB(T)$. In this case, $T$ contains an induced $P_4$, namely $z_1uxv$, contradicting that $T$ is $\indsat{P_4}$. By a symmetric argument, we may assume $xz_2\in EB(T)$.

We will partition the vertices of $X$ into equivalence classes according to their neighborhoods in $Y$.  For each $x\in X$, denote $N_Y^B(x)$ to be the set of $y\in Y$ such that the edge $xy$ is black. In order to proceed, we need two claims. Claim~\ref{cl:betweenxi} establishes that the equivalence classes have the same color edge between them.
\begin{claim}\label{cl:betweenxi}
If $x,x'\in X$ have different neighborhoods in $Y$, then $xx'$ is black.

If $y,y'\in Y$ have different neighborhoods in $X$, then $yy'$ is white.
\end{claim}
\begin{proofcite}{Claim~\ref{cl:betweenxi}}
Suppose $xx'$ is not black. If $\tilde{y}\in Y$ where $x\tilde{y}$ is black and $x'\tilde{y}$ is white, then $\tilde{y}xux'$ is a realization of $P_4$.

Suppose $yy'$ is not white. If $\tilde{x}\in X$ where $y\tilde{x}$ is black and $y'\tilde{x}$ is white, then $u\tilde{x}yy'$ is a realization of $P_4$. This proves Claim~\ref{cl:betweenxi}.
\end{proofcite}~\\

Claim~\ref{cl:equiv} establishes that the black neighborhoods in $Y$ nest.
\begin{claim}\label{cl:equiv}
For all $x,x'\in X$, either $N_Y^B(x)\subseteq N_Y^B(x')$ or $N_Y^B(x)\supseteq N_Y^B(x')$.

For all $y,y'\in Y$, either $N_X^B(y)\subseteq N_X^B(y')$ or $N_X^B(y)\supseteq N_X^B(y')$.
\end{claim}
\begin{proofcite}{Claim~\ref{cl:equiv}}
Suppose $\tilde{y}\in N_Y^B(x)-N_Y^B(x')$ and $\tilde{y}'\in N_Y^B(x')-N_Y^B(x)$. Using Claim~\ref{cl:betweenxi}, $xx'$ is black and $\tilde{y}\tilde{y}'$ is white. Hence, $\tilde{y}xx'\tilde{y}'$ is a realization of $P_4$.

Suppose $\tilde{x}\in N_X^B(y)-N_X^B(y')$ and $\tilde{x}'\in N_X^B(y')-N_X^B(y)$. Using Claim~\ref{cl:betweenxi}, $yy'$ is white and $\tilde{x}\tilde{x}'$ is black. Hence, $y\tilde{x}\tilde{x}'y'$ is a realization of $P_4$. This proves Claim~\ref{cl:equiv}.
\end{proofcite}~\\

So, we define an equivalence relation on the vertices in $X$ so that $x,x'\in X$ are equivalent if and only if $N_Y^B(x)=N_Y^B(x')$. It is easy to see that this is an equivalence relation. Let the equivalence classes be $X_1,\ldots,X_{\ell}$ with the property that $x_i\in X_i$ and $x_{i+1}\in X_{i+1}$ imply $N_Y^B(x_i)\subset N_Y^B(x_{i+1})$, for $i=1,\ldots,\ell-1$.

By definition, $X_{\ell}\neq\emptyset$ so let $x\in X_{\ell}$ and consider some $y\in Y$ such that $y\not\in N_Y^B(x)$ (implying $y$ has no neighbors in $X$). If $z_1\in Z_1$ such that $yz_1\in EB(T)$, then since $xz_1$ is black, we have an induced $P_4$, namely $vxz_1y$, a contradiction. But, if $y$ has no neighbors in $Z_1$, then $y$ and $u$ can have no common neighbors.  Therefore, ${\rm dist}(y,u)\geq 3$.

If ${\rm dist}(y,u)<\infty$, then there is an induced path on at least 4 vertices between $y$ and $u$, giving the existence of an induced $P_4$. If ${\rm dist}(y,u)=\infty$, then either $T$ is a trigraph with all white edges (hence not $\indsat{P_4}$) or $T$ has at least one nontrivial component.  By Fact~\ref{CompleteDisjoint}, each component is $\indsat{P_4}$ and so applying the inductive hypothesis to any nontrivial component gives a gray edge.

Finally, we may conclude that the edges in $T[X_{\ell},V-X_{\ell}]$ are all black. By Fact~\ref{CompleteDisjoint}, both $X_{\ell}$ and $V-X_{\ell}$ are $\indsat{P_4}$. Applying the inductive hypothesis to whichever of those sets is nontrivial gives a gray edge and a contradiction.  This concludes the proof of Lemma~\ref{LemGrayExists}.~\hfill~$\Box$~\\


\subsubsection{Proof of Lemma~\ref{TechnicalLemma}}
\label{sec:prooftechlemma}
We wish to take note that this lemma requires the strong inductive hypothesis from the proof of Theorem~\ref{TheoremP4Smallest}.

If either $X$ or $Y$ is empty, then the other set has at least two vertices and the inductive hypothesis gives that there are at least $\left\lceil\frac{|X|+|Y|+1}{3}\right\rceil$ gray edges.

As in the proof of Lemma~\ref{LemGrayExists}, we will partition the vertices of $X$ into equivalence classes according to their neighborhoods in $Y$.  The statements of Claim~\ref{cl:betweenxi} and~\ref{cl:equiv} were written to apply to this lemma as well, even though gray edges are permitted in $T[X]$ and in $T[Y]$.  So, the edges between equivalence classes of $X$ are black and between equivalence classes of $Y$ are white.  Furthermore, the sets $N_Y^B(x)$ -- the black neighborhood of $x$ in set $Y$ -- form a nesting family.

So, again define an equivalence relation on the vertices in $X$ so that $x,x'\in X$ are equivalent if and only if $N_Y^B(x)=N_Y^B(x')$. Let the equivalence classes be $X_1,\ldots,X_{\ell}$.  This, in turn, defines a partition of $Y$.  For $j=1,\ldots,\ell+1$, the set $Y_j$ has the property that the edges in $T[X_i,Y_j]$ are white for $i=1,\ldots,j-1$ and the edges in $T[X_i,Y_j]$ are black for $i=j,\ldots,\ell$. See Figure~\ref{fig:XYequiv}.

\begin{figure}[ht]
  \centering
  \includegraphics[height=\figheight]{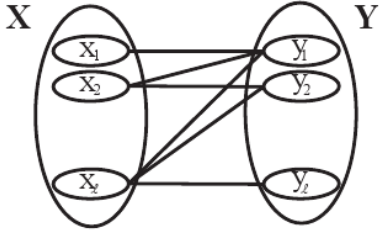}
  \caption{A decomposition of $X$ and $Y$ into equivalence classes. The pair $(X_i,Y_j)$ consists of black edges if $i\geq j$; otherwise, it consists of white edges.}\label{fig:XYequiv}
\end{figure}

The sets $Y_2,\ldots,Y_{\ell}$ are nonempty because if $Y_j$ is empty, then $N_Y^B(x)=Y_1\cup\cdots\cup Y_{j-1}$ for all $x\in X_{j-1}\cup X_j$, contradicting the definition of the equivalence classes. By Fact~\ref{FactSameBehavior}, each of the sets $X_1,\ldots,X_{\ell},Y_1,\ldots,Y_{\ell+1}$ is $\indsat{P_4}$. Thus, if the sets have size at least two, there will be enough gray edges to prove the lemma.  We will ensure that not too many of the sets are of size one.
\begin{claim}\label{cl:smallsets}
For $i=1,\ldots,\ell$, it cannot be the case that both $|X_i|=1$ and $|Y_i|=1$ or that $|X_i|=1$ and $|Y_{i+1}|=1$.
\end{claim}
\begin{proofcite}{Claim~\ref{cl:smallsets}}
First, suppose that $|X_i|=|Y_i|=1$ and let $X_i=\{x_i\}$ and $Y_i=\{y_i\}$. Consider $\flip{T}{x_iy_i}$ and note we may now regard $x_iy_i$ as a white edge because it was black in $T$ itself. There must be a black neighbor of $y_i$ in the realization and that must be $x_j\in X_j$ for some $j>i$. (See the diagram in Figure~\ref{fig:XYequiv}.) But $x_j$ is a black neighbor of $x_i$ also. So, $x_ix_jy_i$ is a realization of $P_3$. Suppose there is an additional vertex, $w$, in the realization of the $P_4$. If $w\in X-X_j$, then $wx_ix_j$ induces a black triangle and if $w\in X_j$ then $wx_ix_jy_i$ has a black $C_4$. Both of these cases produce a contradiction. If $w\in Y$, then $wy_i$ is white and being in the realization of $P_4$ requires $wx_i$ to be black. However, this means $wx_j$ is also black because $j>i$. This is also a contradiction.  Hence there is no realization of $P_4$ in $\flip{T}{x_iy_i}$.

Second, suppose that $|X_i|=|Y_{i+1}|=1$ and let $X_i=\{x_i\}$ and $Y_{i+1}=\{y_{i+1}\}$. Consider $\flip{T}{x_iy_{i+1}}$ and note we may now regard $x_iy_{i+1}$ as a black edge because it was white in $T$ itself. There must be a white neighbor of $x_i$ in the realization and that must be $y_j\in Y_j$ for some $j>i+1$. But $y_j$ is a white neighbor of $y_{i+1}$ also. So, $x_iy_jy_{i+1}$ is a realization of $\overline{P_3}$. Suppose there is an additional vertex, $w$, in the realization of the $P_4$. If $w\in Y-Y_j$, then $wy_{i+1}y_j$ induces a white triangle and if $w\in Y_j$ then $wx_iy_jy_{i+1}$ has a white $C_4$. Both of these cases produce a contradiction. If $w\in X$, then $wx_i$ is black and being in the realization of $P_4$ requires $wy_j$ to be black. However, this means $wy_{i+1}$ is also black because $j>i+1$. This is also a contradiction.  Hence there is no realization of $P_4$ in $\flip{T}{x_iy_{i+1}}$.

This proves Claim~\ref{cl:smallsets}.
\end{proofcite}~\\

With Claim~\ref{cl:2cases}, the proof of the lemma is almost finished, leaving only two exceptional cases.
\begin{claim}\label{cl:2cases}
There are at least $\left\lceil\frac{|X|+|Y|}{3}\right\rceil$ gray edges unless one of the following cases occurs:
\begin{enumerate}[(i.)]
\item $|Y_1|=\cdots=|Y_{\ell+1}|=1$ and $|X_1|,\ldots,|X_{\ell}|\geq 2$.\label{it:y1}
\item $|X_1|=\cdots=|X_{\ell}|=1$ and $|Y_2|,\ldots,|Y_{\ell}|\geq 2$ but $Y_1=Y_{\ell+1}=\emptyset$.\label{it:x1}
\end{enumerate}
\end{claim}
\begin{proofcite}{Claim~\ref{cl:2cases}}
First suppose $Y_1\neq\emptyset$.  Consider the pairs $(X_i,Y_i)$ for $i=1,\ldots,\ell$. By Claim~\ref{cl:smallsets}, at least one of the sets must have size at least 2.  If both have size at least 2, then the number of gray edges in $X_i\cup Y_i$ is at least $\left\lceil\frac{|X_i|+1}{3}\right\rceil+\left\lceil\frac{|Y_i|+1}{3}\right\rceil \geq\left\lceil\frac{|X_i|+|Y_i|+2}{3}\right\rceil$. If not, say $|Y_i|\geq 2$, then the number of gray edges in $X_i\cup Y_i$ is at least $\left\lceil\frac{|Y_i|+1}{3}\right\rceil=\left\lceil\frac{|X_i|+|Y_i|}{3}\right\rceil$. So, in this case, the total number of gray edges in $X\cup Y$ is at least $\left\lceil\frac{|X|+|Y|}{3}\right\rceil$ unless $|Y_{\ell+1}|=1$ and there are $\ell$ other components of size $1$.  By Claim~\ref{cl:smallsets}, this can only occur if $|Y_1|=\cdots=|Y_{\ell}|=1$. This is case~(\ref{it:y1}.).

Second, suppose $Y_1=\emptyset$.  Consider the pairs $(X_i,Y_{i+1})$ for $i=1,\ldots,\ell$. By Claim~\ref{cl:smallsets}, at least one of the sets must have size at least 2.  If both have size at least 2, then the number of gray edges in $X_i\cup Y_{i+1}$ is at least $\left\lceil\frac{|X_i|+1}{3}\right\rceil+\left\lceil\frac{|Y_{i+1}|+1}{3}\right\rceil \geq\left\lceil\frac{|X_i|+|Y_{i+1}|+2}{3}\right\rceil$. If not, say $|Y_{i+1}|\geq 2$, then the number of gray edges in $X_i\cup Y_{i+1}$ is at least $\left\lceil\frac{|Y_{i+1}|+1}{3}\right\rceil=\left\lceil\frac{|X_i|+|Y_{i+1}|}{3}\right\rceil$. So, in this case, the total number of gray edges in $X\cup Y$ is at least $\left\lceil\frac{|X|+|Y|}{3}\right\rceil$ unless $|X_1|=1$ and there are $\ell-1$ other components of size $1$.  By Claim~\ref{cl:smallsets}, this can only occur if $|X_2|=\cdots=|X_{\ell}|=1$. This is case~(\ref{it:x1}.).

Therefore, the only cases that remain are case~(\ref{it:y1}.) and ~(\ref{it:x1}.), completing the proof of Claim~\ref{cl:2cases}.
\end{proofcite}~\\

\noindent\textbf{Case~(\ref{it:y1}.)} Consider the trigraph, $T'$, induced by $V(T)-\left(X_{\ell}\cup Y_{\ell+1}\right)$. We claim that $T'$ is $\indsat{P_4}$. Suppose not and consider $\flip{T}{e}$ such that $e$ is a white or black edge with both endpoints in $V(T')$. No realization of $P_4$ in $T$ can have exactly three of its vertices in $V(T')$ because the pair $(Y_{\ell+1},V(T'))$ has only white edges, giving a vertex of degree 0, and $(X_{\ell},V(T'))$ has only black edges, giving a vertex of degree 3. No realization of $P_4$ in $T$ can have two vertices in $V(T')$ and two in $X_{\ell}$, which would give a $C_4$. Finally, the realization cannot have $x_{\ell}\in X_{\ell}$ and the vertex in $Y_{\ell+1}$, making $x_{\ell}$ of degree 3. Therefore, the realization must have all four vertices in $V(T')$.

Hence, $V(T')$ is a $\indsat{P_4}$ trigraph on at least $2$ vertices and by the inductive hypothesis, the number of gray edges in $T$ is at least
$$ \left\lceil\frac{|V(T')|+1}{3}\right\rceil+\left\lceil\frac{|X_{\ell}|+1}{3}\right\rceil \geq\left\lceil\frac{|V(T')|+|X_{\ell}|+2}{3}\right\rceil =\left\lceil\frac{n+1}{3}\right\rceil . $$~\\

\noindent\textbf{Case~(\ref{it:x1}.)} Consider the trigraph, $T''$, induced by $V(T)-\left(X_{\ell}\cup Y_{\ell}\right)$. We claim that $T''$ is $\indsat{P_4}$. Suppose not and consider $\flip{T}{e}$ such that $e$ is a white or black edge with both endpoints in $V(T'')$. No realization of $P_4$ in $T$ can have exactly three of its vertices in $V(T'')$ because the pair $(X_{\ell},V(T''))$ has only black edges, giving a vertex of degree 3, and $(Y_{\ell},V(T''))$ has only white edges, giving a vertex of degree 0. No realization of $P_4$ in $T$ can have two vertices in $V(T'')$ and two in $Y_{\ell}$, which would give a $\overline{C_4}$. Finally, the realization cannot have a vertex in $Y_{\ell}$ and the vertex $x_{\ell}$ in $X_{\ell}$, making $x_{\ell}$ of degree 3. Therefore, the realization must have all four vertices in $V(T'')$.

Hence, $V(T'')$ a $\indsat{P_4}$ trigraph on at least $2$ vertices and by the inductive hypothesis, the number of gray edges in $T$ is at least
$$ \left\lceil\frac{|V(T'')|+1}{3}\right\rceil+\left\lceil\frac{|Y_{\ell}|+1}{3}\right\rceil \geq\left\lceil\frac{|V(T'')|+|Y_{\ell}|+2}{3}\right\rceil =\left\lceil\frac{n+1}{3}\right\rceil . $$~\\

This concludes the proof of Lemma~\ref{TechnicalLemma}.~\hfill~$\Box$~\\


\section{Proofs of Facts}
\label{sec:factproofs}

 In this section we will provide proofs to the general facts stated earlier.~\\


\begin{proofcite}{Fact~\ref{SelfComplementary}} Let $T$ be a $\indsat{P_4}$ trigraph. We want to show that its complement is also $\indsat{P_4}$.  If $\comp{R}$ is a realization of $\comp{T}$ with an induced $P_4$, then $R$ is a realization of $T$ with an induced $P_4$.  Thus, $\comp{T}$ has no realization with an induced $P_4$.  Similarly, if $R$ is a realization of $\flip{T}{uv}$ with an induced $P_4$, then $\comp{R}$ is a realization of $\flip{\comp{T}}{uv}$ with an induced $P_4$. So, $\comp{T}$ must be $\indsat{P_4}$. This proves Fact~\ref{SelfComplementary}.
\end{proofcite}~\\


\begin{proofcite}{Fact~\ref{CompleteDisjoint}}  Using Fact \ref{SelfComplementary}, and without loss of generality, we may assume that all edges in $T[V_1,V_2]$ are white. We only need to show that $T[V_1]$ is $\indsat{P_4}$, $T[V_2]$ follows identically. If $T[V_1]$ is a gray complete graph or a single vertex, we are finished because a single vertex is trivially $\indsat{P_4}$.  So let us assume that $T[V_1]$ is not a gray complete graph. Let $e$ be a black or white edge in $T[V_1]$.  Now, consider $P$, a realization of an induced $P_4$ in $\flip{T}{e}$.  As there are only white edges between $V_1$ and $V_2$, it must be that no vertex of $P$ can be in $V_2$.  So, $T[V_1]$ is $\indsat{P_4}$.  This proves Fact~\ref{CompleteDisjoint}.
\end{proofcite}~\\


\begin{proofcite}{Fact~\ref{FactSameBehavior}} Let none of the edges in $T[S,V-S]$ be gray.  Let $s_1s_2$ be a black or white edge in $S$. Fact~\ref{SelfComplementary} allows us to assume that, without loss of generality, $s_1s_2$ is white.  Consider $P$, a realization of an induced $P_4$ in $\flip{T}{s_1s_2}$.  We will show that $P$ can not contain any vertices from $V-S$.

We consider two cases. First, suppose $P$ has exactly one $v\in V-S$ and the fourth vertex is $s_3\in S$.  Since every pair of vertices in $S$ has the same neighborhood in $V-S$, either $vs_i$ is white for $i=1,2,3$ or $vs_i$ is black for $i=1,2,3$.  This gives a vertex in $P$ either with degree 0 or degree 3, hence $P$ cannot be an induced $P_4$.

Second, suppose $P$ has exactly two vertices $v_1,v_2\in V-S$. Since the neighborhoods of $s_1$ and $s_2$ are the same in $V-S$, either (1) $s_iv_j$ is black for $i,j\in\{1,2\}$ or (2) $s_iv_j$ is white for $i,j\in\{1,2\}$ or, without loss of generality, (3) $s_iv_1$ is black for $i\in\{1,2\}$ and $s_iv_2$ is white for $i\in\{1,2\}$.

If (1) occurs, then $P$ has a $C_4$, a contradiction. If (2) occurs, then $P$ has a $\comp{C_4}$, a contradiction.  If (3) occurs, then either $v_1$ has degree 3 in $P$ or $v_2$ has degree 0 in $P$. In any case, we have a contradiction. Thus, $P$ must have all $4$ vertices in $S$ and so $T[S]$ is $\indsat{P_4}$.  This proves Fact~\ref{FactSameBehavior}.
\end{proofcite}~\\


\begin{proofcite}{Fact~\ref{FactNoGrayP4}} This fact is equivalent to saying $T$ cannot have a gray $P_4$. For the sake of contradiction, suppose that it does. By Fact~\ref{SelfComplementary}, we may assume, without loss of generality, that there are 4 vertices that form a gray $P_4$ and have either zero black edges or one black edge.  Let $abcd$ be the gray $P_4$.

If there are no black edges in $T[\{a,b,c,d\}]$, then $abcd$ is a realization of $P_4$. If $ac$ is the sole black edge, then $bacd$ is a realization of $P_4$. If $bd$ is the sole black edge, then $abdc$ is a realization of $P_4$. Finally, if $ad$ is the sole black edge, then $badc$ is a realization of $P_4$.

In all cases, there is a realization of $P_4$ in $T$, a contradiction to $T$ being $\indsat{P_4}$.  This proves Fact~\ref{FactNoGrayP4}.
\end{proofcite}~\\


\begin{proofcite}{Fact~\ref{StarBehavior}}
Let $T$ be a trigraph which is $\indsat{P_4}$ such that $C=\{u,v_1,\ldots,v_{k-1}\}$ is a gray star with center $u$.

First, we will show that the edges in $T[\{v_1,\ldots,v_{k-1}\}]$ (i.e., the edges between the leaves) have the same color. If $k\geq 4$ and $v_1v_2$, $v_1v_3$ are white and $v_2v_3$ is black, then $v_1uv_2v_3$ is a realization of $P_4$, a contradiction.  See Figure~\ref{grayS4}. If $k\geq 4$ and $v_1v_2, v_1v_3$ are black and $v_2v_3$ is white, then $uv_3v_1v_2$ is a realization of $P_4$, a contradiction.  Since $v_1,v_2,v_3$ could be chosen arbitrarily, $T[\{v_1,\ldots,v_{k-1}\}]$ is monochromatic.

\begin{figure}[ht]
   \centering
  \includegraphics[height=\figheight]{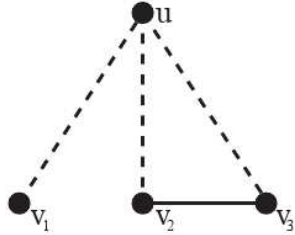}
   \caption{An example for the proof of Fact \ref{StarBehavior}. A gray star with induced $P_4$ namely $v_1uv_2v_3$.}\label{grayS4}
\end{figure}

Second, we will show that, for every vertex $w\in V(T)-C$, the edges $wv_i$ have the same color. If $k\geq 3$ and $wv_1$ is black and $wv_2$ is white, then either $wv_1uv_2$ or $wv_1v_2u$ is a realization of $P_4$, depending on the color of $v_1v_2$. Since $v_1,v_2$ could be chosen arbitrarily, $T[\{w\},C-\{u\}]$ is monochromatic for any $w$.

Let $X=\{x\in V(T)-C : T[\{x\},C]\mbox{ is black}\}$, let $Y=\{y\in V(T)-C : T[\{y\},C]\mbox{ is white}\}$ and let $Z=V(T)-\left(X\cup Y\cup C\right)$. For each $z\in Z$, either both the edge $zu$ is white and $T[\{z\},C-\{u\}]$ is black or both the edge $zu$ is black and $T[\{z\},C-\{u\}]$ is white.

It remains to show that the vertices in $Z$ behave as set forth in the fact.  If $k\geq 3$, we may assume, without loss of generality, that the edges of $T[\{v_1,\ldots,v_{k-1}\}]$ are white.  If $z\in Z$ has the property that $zu$ is white and $T[\{z\},C-\{u\}]$ is black, then $v_1zv_2u$ is a realization of $P_4$, a contradiction.  Thus, all $z\in Z$ must have the property that $T[\{z\},C-\{u\}]$ has the same color as $T[\{v_1,\ldots,v_{k-1}\}]$ and $zu$ is the complementary color.  This is exactly the condition given in the statement of the fact.

If $k=2$, then we just need to verify that $z_1,z_2\in Z$ have the same neighborhood in $\{u,v_1\}$.  If they do not, then we may assume that $z_1u$ and $z_2v_1$ are black and $z_1v_1$ and $z_2u$ are white.  In this case, either $z_1uv_1z_2$ or $uz_1z_2v_1$ is a realization of $P_4$ depending on the color of $z_1z_2$.

This classifies all the vertices in $V(T)$ and proves Fact~\ref{StarBehavior}.
\end{proofcite}~\\


\begin{proofcite}{Fact~\ref{TriangleBehavior}} Let $T$ be a trigraph which is $\indsat{P_4}$ such that $T[\{v_1,v_2,v_3\}]$ is a gray triangle. If Fact~\ref{TriangleBehavior} fails to hold, then there exists an $x\in V(T)-\{v_1,v_2,v_3\}$ such that, without loss of generality, $v_1x$ is black and $v_2x, v_3x$ are white.  But in this case, $T$ has a realization of $P_4$, namely $xv_1v_2v_3$, as is seen in Figure~\ref{grayK3}. This proves Fact~\ref{TriangleBehavior}.
\begin{figure}[ht]
  \centering
  \includegraphics[height=\figheight]{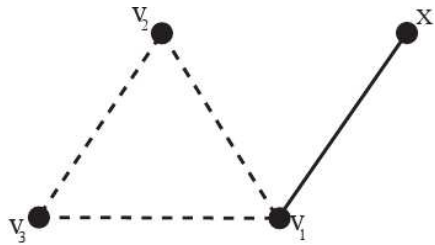}
  \caption{Example for proof of Fact \ref{TriangleBehavior}. $T$ contains a realization of $P_4$, namely $xv_1v_2v_3$.}\label{grayK3}
\end{figure}
\end{proofcite}~\\

\section{Conclusion}
\label{sec:conc}
It is not clear whether our constructions of $n$-vertex $\indsat{P_4}$ trigraphs with $\lceil (n+1)/3\rceil$ gray edges are unique (up to complementation) for some cases of $n$.  For instance, suppose $n$ is divisible by $3$ and $T$ is constructed as follows: $V(T)=Z\cup C_0\cup\{y\}$ where $C_0=\{u,v_1,v_2\}$ induces a gray star on $3$ vertices with a black edge $v_1v_2$, the edges in $T[Z,\{v_1,v_2\}]$ are black, $T[Z]$ is a $\indsat{P_4}$ trigraph on $n-4$ vertices with $\lceil (n-3)/3\rceil$ gray edges and all other edges in $T$ are white.  This $T$ has $\lceil (n+3)/3\rceil=\lceil (n+1)/3\rceil$ gray edges.

The authors would like to thank to Mike Ferrara and Ron Gould for helpful conversations regarding graph saturation. The authors would like to thank anonymous reviewers whose careful reading improved the proofs and the paper, in general.

\bibliographystyle{plain}
\bibliography{ResearchBib}

\end{document}